%% file: mainfile.tex
\documentclass[a4paper,twoside,12pt,preprint]{article}

\usepackage[english]{babel}
\usepackage[T1]{fontenc}
\usepackage[ansinew]{inputenc}
\usepackage{geometry}
\usepackage{color} 
 \usepackage{amsmath}
 \numberwithin{equation}{section}
\usepackage{lmodern}

\usepackage{graphicx} 

\usepackage{amsmath}
\usepackage{amsthm}
\usepackage{amsfonts}
\usepackage{amssymb}
\usepackage{graphics}
\usepackage{color}

\newtheorem{theorem}{Theorem}
\newtheorem{proposition}[theorem]{Proposition}
\newtheorem{lemma}[theorem]{Lemma}
\newtheorem{corollary}[theorem]{Corollary}
\newtheorem{satz}{Theorem}

\theoremstyle{remark}

\newtheorem{remark}{Remark}

\definecolor{orange}{rgb}{1,0.5,0}

\usepackage{enumerate}

\newcommand{\IN}{\mathbb{N}} 

\newcommand{\lmr}{\text{lmr}} 

\newcommand{\lt}{\underline{t}} 
\newcommand{\gt}{\overline{t}} 
\newcommand{\ltau}{\underline{\tau}} 
\newcommand{\gtau}{\overline{\tau}} 

\newcommand{\D}{\mathbb{D}}

\newcommand{\N}{\mathbb{N}}
\newcounter{cntselflist}
\newenvironment{selflist}{\begin{list}{\mbox{}\hspace{0.5cm}\mbox{}\arabic{cntselflist})\mbox{ }}{
	\setlength\itemindent{0pt}
	\setlength\leftmargin{0pt}
	\setlength\itemsep{3pt}
	\setlength\labelwidth{0pt}
	\setlength\labelsep{0pt}
	\usecounter{cntselflist}
	}  }{\end{list}}
\begin{document}
\parindent 0pt

\setcounter{section}{0}

	 \author{Christoph B\"ohm and Sebastian Schlei\ss inger}
   \title{Constant Coefficients in the Radial Komatu-Loewner Equation for Multiple Slits}
   \date{\today}
   \maketitle

\begin{abstract}
      The radial Komatu-Loewner equation is a differential equation for certain normalized conformal mappings that can be used to describe the growth of slits within multiply connected domains. We show that it is possible to choose constant coefficients in this equation in order to generate given disjoint slits and that those coefficients are uniquely determined under a suitable normalization of the differential equation.
\end{abstract}





	\maketitle
	
	\input{Input/Chapter1}
	\input{Input/Chapter2}
	\input{Input/Chapter3}
	\input{Input/Chapter4}

	\bibliography{Input/literature}{}
	\bibliographystyle{amsplain}

\end{document}

%% file: Input/Chapter1.tex
\section{Introduction and results}

In 1923, C. Loewner derived a differential equation for a certain family of conformal mappings to 
attack the Bieberbach conjecture,
see \cite{Loewner:1923}. Loewner's method has been extended and turned out to be a useful tool 
within complex analysis. In particular, the Loewner differential equations provide a powerful 
tool for the description of the growth of slits in a given planar domain. 
After O. Schramm discovered Stochastic Loewner Evolution (or Schramm Loewner Evolution, SLE) 
in \cite{MR1776084}, it became clear that those models have many applications in different 
mathematical and physical disciplines, especially in statistical physics.\\

In the classical setting, Loewner theory describes the evolution of a family of simply connected, 
proper subsets of the complex plane $\mathbb{C},$ which are all conformally equivalent to the 
unit disc $\D:=\{z\in\mathbb{C} \;|\; |z|<1\}$ according to the Riemann Mapping Theorem. 
In 1943, Y. Komatu showed that Loewner's ideas are not confined to the simply connected case: 
He derived a Loewner equation for the growth of a slit within a doubly connected domain, 
see \cite{KomatuZweifach}. In \cite{Komatu}, he considered a generalization of Loewner's 
differential equation to a more general finitely connected domain. Komatu's ideas have been applied and 
extended by several authors. Recently, R. Bauer and R. Friedrich derived a radial and a 
chordal Komatu-Loewner equation to deal with the growth of a (stochastic) slit in a multiply 
connected domain, see \cite{BauerFriedrichCSD} and \cite{BauerFriedrichCBC}. In the radial
setting, the slit grows within a \textit{circular slit disk} $D$, i.e. 
$D= \D \setminus (C_1\cup...\cup C_N)$, $N\in \N_0,$ where each 
$C_j\subset\D$ is a circular arc centered at $0$ such that 
$C_j\cap C_k=\emptyset$ whenever $j\not=k.$ Note that every $N$-connected domain $\Omega$ can 
be mapped onto such a circular slit disk $D$ by a conformal map $f: \Omega \rightarrow D.$
This mapping is unique if we require the normalization $f(z_0)=0,$ $f'(z_0)>0$ for some 
$z_0\in \Omega,$ see \cite{ConwayII}, Chapter 15.6.\\

In \cite{BoehmLauf}, W. Lauf and the first author generalized the radial Komatu-Loewner 
equation for the growth of several slits:\\
Let $\Omega$ be an arbitrary circular slit disk and  let 
$\gamma_1,...,\gamma_m:[0,T]\to\overline{\Omega}$ be parametrizations of pairwise disjoint simple 
curves such that $\gamma_k(0)\in\partial \D$ and $\gamma_k(0,T]\subset 
\Omega\setminus\{0\}$ for all $k=1,...,m.$ \\
Furthermore, if $D$ is a circular slit disk and $u\in\partial \D$, we denote by 
$w\mapsto \Phi(u,w;D)$ the unique conformal mapping from $D$ onto the right half-plane minus 
slits parallel to the imaginary axis with $\Phi(u,u;D)=\infty$ and $\Phi(u,0;D)=1.$\\
We summarize one of the main results of \cite{BoehmLauf} in the following theorem:
\begin{satz}[Corollary 5 in \cite{BoehmLauf}] \label{The:KomLowEqu}
	Let $\Omega_t=\Omega \setminus \bigcup_{k=1}^m \gamma_j[0,t]$ and denote by $g_t$ the unique 
	conformal mapping $g_t: \Omega_t \to D_t$ where $D_t$ is a circular slit disk 
	and $g_t(0)=0,$ $g'_t(0)>0.$\\
	Then there exists a Lebesgue measure zero set $\mathcal{N}$ such that for every 
	$z\in \Omega_T$ the function $t\mapsto g_t(z)$ is differentiable on 
	$[0,T]\setminus \mathcal{N}$ with 
	\begin{equation}\label{KLE} 
		\dot{g}_t(z) = g_t(z) \sum_{k=1}^m \lambda_k(t)\cdot \Phi(\xi_k(t),g_t(z);D_t),
	\end{equation}
	where the continuous function $t\mapsto\xi_k(t)\in \partial\D$ is the image of 
	$\gamma_k(t)$ under the map $g_t$ and the coefficient functions $t\mapsto \lambda_k(t)$ are measurable with $\lambda_k(t)\geq 0$ for every 
	$t\in[0,T]$.
\end{satz}
\begin{remark} The continuous functions 
	$\xi_k:[0,T]\to\partial\D$ are usually called \emph{driving functions}. \\
	Informally, the coefficient function $\lambda_k(t)$ corresponds to the speed
	of growth of the slit parametrized by $\gamma_k.$ From the normalization $\Phi(u,0;D)=1$ it follows that $g'_t(0)=e^{\int_0^t \sum_{k=1}^m \lambda_k(\tau)\; d\tau}.$ \\

\end{remark}

\begin{figure}[h]
	\centering \includegraphics[width=135mm]{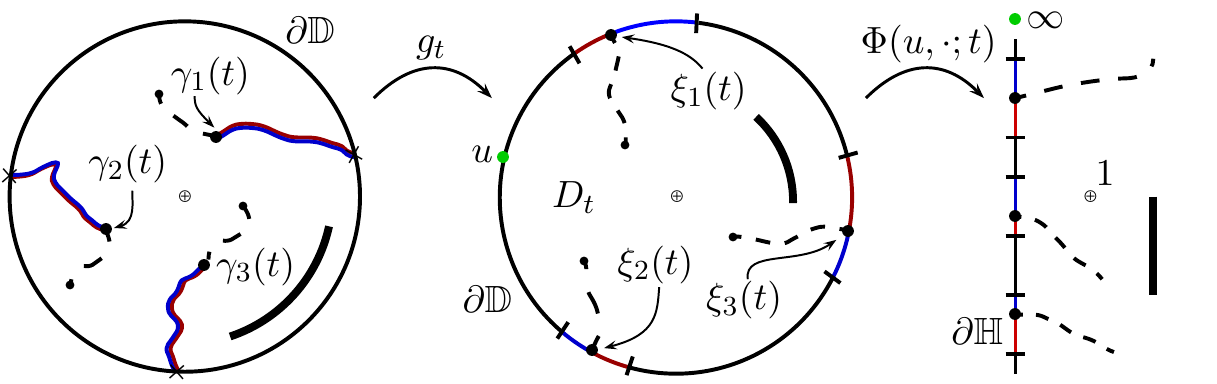}
	\caption{The mappings $z\mapsto g_t(z)$ and $w\mapsto \Phi(u,w,D_t)$ from Theorem \ref{The:KomLowEqu}.}
\end{figure}

Note that there are no further assumptions on the parametrizations of the slits 
$\Gamma_k:=\gamma_k[0,T]$ in Theorem \ref{The:KomLowEqu}. \\
Now suppose we are given only the circular slit disk $\Omega$ and the slits 
$\Gamma_1,...,\Gamma_m$ without parametrization. Roughly speaking, we address the question if it is possible to find a simple form of equation (\ref{KLE}) such that it has still enough parameters to generate the slits $\Gamma_1,...,\Gamma_m$, but, on the 
other hand, the choice of those parameters is unique. We will see that this is possible and, 
moreover, it will turn out that equation (\ref{KLE}) is satisfied for \textit{all} $t$ 
in this case.

 The latter can be interpreted as a generalization of Loewner's original idea of finding a parametrization of an arbitrary curve such that the family of certain associated conformal mappings is differentiable. In some sense, this problem is related to Hilbert's fifth problem of finding differentiable structures for continuous groups, see \cite{Goodman}.

\medskip

In a first step, we can choose parametrizations such that the mappings $g_t$ satisfy $g_t'(0)=e^t,$ i.e. $\sum_{k=1}^m \lambda_k(t)\equiv 1.$ However, there are many of such parametrizations when $m\geq 2$.  

 In this work we will show that there exist unique parametrizations of the slits 
$\Gamma_1,...,\Gamma_m$, such that the corresponding Komatu-Loewner equation is satisfied 
for all $t$ where all coefficient functions $\lambda_k(t)$ are constant and sum up to 1.

\medskip

We need one further notation: 
Let $\Omega$ be a circular slit disc and let $f:\Omega\setminus(\Gamma_1\cup\ldots\cup\Gamma_m)
\to D$ be the unique conformal mapping onto a
circular slit disk $D$ with $f(0)=0$ and $f'(0)>0$, then $f'(0)> 1$ and the 
\emph{logarithmic mapping radius} of $\Omega\setminus(\Gamma_1\cup\ldots\cup\Gamma_m)$ is 
defined to be the real number $\log f'(0)> 0$. The inequality $f'(0)>1$ is 
an immediate consequence of Lemma \ref{Lem:PropLMR} b). 

\begin{theorem} \label{The:ConCoeffi}
	Let $L$ be the logarithmic mapping radius of $\Omega\setminus(\Gamma_1\cup\ldots\cup\Gamma_m)$. 
	There exist unique continuous parametrizations 
	$$
		\delta_1:[0,L]\to\Gamma_1,..., \delta_m:[0,L]\to\Gamma_m
	$$
	and unique $\lambda_1,\ldots,\lambda_m\in(0,1)$ with $\sum_{k=1}^m\lambda_k=1$ such that 
	the following holds: \\
	Let $\Omega_t=\Omega \setminus \bigcup_{k=1}^m \delta_j[0,t]$ and denote by $h_t$ the unique 
	conformal mapping $h_t: \Omega_t \to D_t$ where $D_t$ is a circular slit disk and 
	$h_t(0)=0,$ $h'_t(0)>0.$\\
	Then $h_t'(0)=e^t$ and for every $z\in \Omega_L,$ the function $t\mapsto h_t(z)$ is 
	differentiable on $[0,L]$ with
	\begin{equation}\label{KLE2}
		\dot{h}_t(z) = h_t(z)\sum_{k=1}^m \lambda_k \cdot \Phi(\xi_k(t),h_t(z);D_t), 
	\end{equation}
	where $\xi_k(t)\in \partial\D$ is the image of $\delta_k(t)$ under the map $h_t$. 
	Moreover, the driving functions $t\mapsto \xi_k(t)$ are continuous.
\end{theorem}

In the simply connected case, we have $D_t=\D$ for all $t$ and equation (\ref{KLE2}) can be 
written down explicitly as 
$$
	\Phi(u,w,\D)=\frac{u+w}{u-w}.
$$

Furthermore, $m$ coefficients $\lambda_1,...,\lambda_m\in(0,1)$ and $m$ continuous driving 
functions $\xi_1,...,\xi_m:[0,L]\to\partial\D$ determine the unique solution to 
equation (\ref{KLE2}) in the simply connected case. 
If the solution generates slit mappings, then the parametrization of these slits are 
uniquely determined by the driving functions and the coefficients.\\

Thus we can formulate the simply connected case of Theorem \ref{The:ConCoeffi} as follows.
\begin{corollary}
	Let  $\Gamma_1,\ldots,\Gamma_m$ be disjoint slits in $\D$ and let $L$ be the logarithmic 
	mapping radius of $\D\setminus(\Gamma_1\cup\ldots\cup\Gamma_m)$. 
	Then there exist unique $\lambda_1,...,\lambda_m\in(0,1)$ and unique continuous 
	driving functions $\xi_1,\ldots,\xi_m:[0,L]\to\partial\D$ such that the solution of the 
	Loewner equation 
	\begin{equation*}
		\dot{h}_t(z) = h_t(z)\sum_{k=1}^m \lambda_k \cdot \frac{\xi_k(t)+h_t(z)}{\xi_k(t)-h_t(z)}, 
		\qquad h_0(z)=z,
	\end{equation*}
	generates the slits $\Gamma_1,\ldots,\Gamma_m,$ i.e. $h_L$ maps 
	$\D\setminus(\Gamma_1\cup\ldots\cup\Gamma_m)$ conformally onto $\D$.
\end{corollary}

\begin{remark}
	In \cite{Prokhorov:1993}, D. Prokhorov has proven the existence and uniqueness of constant 
	coefficients for several slits in the simply connected case under the assumption that all 
	slits are \emph{piecewise analytic}. This theorem forms the basis for Prokhorov's study of 
	extremal problems for univalent functions in \cite{Prokhorov:1993} by using control-theoretic
	methods. 
	Our proof shows that one can drop any assumption on the regularity of the slits in order 
	to generate them with constant coefficients. 
\end{remark}

We shall give the details of the proof only in the case $m=2$, i.e.~for two slits, 
in order to allow a simple notation. The general case of $m \ge 2$ slits can be proved
inductively in exactly the same way.\\
Our proof combines some technical tools from \cite{BoehmLauf} that were used to prove 
Theorem \ref{The:KomLowEqu} and an idea from \cite{RothSchl}, where a similar result was
proven for the so called chordal Loewner equation in the simply connected case.\\ 
The rest of this paper is organized as follows. 
In Section 2 we describe the setting for the proof and cite some technical results 
from \cite{BoehmLauf}.
The proof of Theorem \ref{The:ConCoeffi} is divided into two parts: In Section 3 we prove 
the existence statement and in Section 4 we give the proof of the uniqueness statement 
of Theorem \ref{The:ConCoeffi}. 

%% file: Input/Chapter2.tex
\section{The setting for the proof}

Suppose $\Omega$ is a circular slit disk and $\Gamma_1$, $\Gamma_2$ are disjoint slits, 
i.e. there are continuous, one-to-one functions $\gamma_1, \gamma_2:[0,T]\to\overline{\Omega}$ 
with $\Gamma_k=\gamma_k[0,T]$ such that $\Gamma_1\cap\Gamma_2=\emptyset$, 
$\gamma_k(0)\in\partial \D$ and $\gamma_k(0,T]\subset \Omega\setminus\{0\}$ for $k= 1,2.$ \\
Let the function $g_t$ be defined as in Theorem \ref{The:KomLowEqu}. 
As $g'_t(0)$ is monotonically increasing (see Lemma \ref{Lem:PropLMR}) and $g'_0(0)=1$, 
we can assume without loss of generality that $g'_t(0)=e^t.$ 
Otherwise, we can simultaneously reparameterize the two slits. 
Consequently we have $T=L$, where $L$ denotes the logarithmic mapping radius of 
$\Omega\setminus(\Gamma_1\cup\Gamma_2)$.\\

Furthermore, we assume that $L=1$ in order to simplify some of the notations.\\

For every $t,\tau\in[0,1]$ we let $f_{t,\tau}$ be the unique conformal mapping from 
$\Omega\setminus(\gamma_1[0,t]\cup \gamma_2[0,\tau])$ onto a circular slit disk with 
$f_{t,\tau}(0)=0$ and $f_{t,\tau}'(0)>0,$ and we define
$$ 
	\lmr(t,\tau):=\log(f'_{t,\tau}(0)). 
$$

Now, in order to prove Theorem \ref{The:ConCoeffi}, we have to show that there exist
\begin{itemize}
	\item two uniquely determined increasing homeomorphisms $u,v:[0,1]\to[0,1]$ and
	\item a uniquely determined $\lambda_0\in(0,1),$
\end{itemize}
such that the Komatu-Loewner equation for the slits $\delta_1:=\gamma_1\circ u$ and 
$\delta_2:=\gamma_2\circ v$ is satisfied for all $t\in[0,1]$ with $\lambda_1(t)= \lambda_0$ 
and $\lambda_2(t)=1-\lambda_0$ and $\lmr(u(t),v(t))=t$ for all $t\in[0,1].$\\

First, we summarize some basic properties of $\lmr$ in the following lemma.

\begin{lemma}${}$\label{Lem:PropLMR}
\begin{enumerate}
\item[(a)] \label{Pro:ComConf} The function $\lmr(t,\tau)$ is continuous in $[0,1]^2$.
\item[(b)] \label{Pro:Monfunf} The function $lmr(t,\tau)$ is strictly increasing with respect to $t$ and $\tau$ respectively.
 \item[(c)]  \label{Pro:DifQuoAbs2} For every $\epsilon>0$ there exists a $\delta>0$ so that for all 
	$0\leq \lt < \gt \leq 1$ and $0\leq \ltau <\gtau \leq 1$ with $|\lt-\gt|<\delta$ and $|\ltau-\gtau|<\delta$
	the following holds:
	\[
		1-\epsilon <\frac{\lmr(\lt,\ltau)-\lmr(\gt,\ltau)}{\lmr(\lt,\gtau)-
		\lmr(\gt,\gtau)} < 1+\epsilon.
	\]
\end{enumerate}
\end{lemma}
\begin{proof}
	See Proposition 6, 8 and 15 in \cite{BoehmLauf}.
\end{proof}

%

Furthermore, we will use a dynamic interpretation of the coefficient functions from 
Theorem \ref{The:KomLowEqu}. Let $a,b:[0,1]\to[0,1]$ be arbitrary strictly increasing
homeomorphisms and let $Z=\{t_0,\ldots,t_s\}$ be a partition of the interval $[0,t]$ 
for a fixed $t\in(0,1]$, i.e. $0=t_0<t_1<...<t_s=t.$ We will denote by 
$|Z|:= \max_{j\in\{0,\ldots,s-1\}}(t_{j+1}-t_j)$ the norm of $Z.$ Now we define the two sums
\begin{align*}
	S_1(a,b,t,Z)&:= \sum_{l=0}^{s-1} \big[\lmr(a(t_{l+1}),b(t_{l})) - \lmr(a(t_{l}),b(t_l))\big],\\
	S_2(a,b,t,Z)&:= \sum_{l=0}^{s-1} \big[\lmr(a(t_l),b(t_{l+1})) - \lmr(a(t_{l}),b(t_l))\big].
\end{align*}

The following proposition relates the limit of $S_j$ for $|Z|\to0$ to the coefficient 
functions $\lambda_j$ of Theorem \ref{The:KomLowEqu}.

\begin{proposition} \label{Pro:ProCalLam}
	Let $a,b:[0,1]\to[0,1]$ be two increasing self-homeomorphisms and let $g_t(z)$ 
	denote the conformal mapping for the slits $\gamma_1\circ a$ and 
	$\gamma_2\circ b$ from Theorem \ref{The:KomLowEqu}. Assume that
	$g'_t(0)=e^t$, i.e. $\lmr(a(t),b(t))=t,$ for all $t\in[0,1]$. Then the limits 
	\[
		c_1(t):=\lim_{|Z|\rightarrow 0} S_1(a,b,Z,t)\quad \text{and} 
		\quad c_2(t):=\lim_{|Z|\rightarrow 0} S_2(a,b,Z,t)
	\]
	exist and form two increasing and Lipschitz continuous functions
	$c_1,c_2:[0,1]\rightarrow[0,\infty)$ with $c_1(0)=c_2(0)=0$ and $c_1(t)+c_2(t)=t$.
	Furthermore, if $c_j$ is differentiable in $t_0$ for $j=1$ and $j=2,$ then 
	the differential equation (\ref{KLE}) holds for $g_{t_0}$ with 
	$$
		\lambda_j(t_0) = \dot c_j(t_0).
	$$
	In this case, $\lambda_1(t_0)$ is equal to the first derivative of the 
	function $t\mapsto\lmr(a(t),b(t_0))$ in $t_0$ and
	$\lambda_2(t_0)$ is equal to the first derivative of the function 
	$t\mapsto\lmr(a(t_0),b(t))$ in $t_0$.
\end{proposition}
\begin{proof}
	This follows immediately from Proposition 16 in \cite{BoehmLauf}, Theorem 2 
	in \cite{BoehmLauf} and the definition of $\lambda_j$ in \cite{BoehmLauf}.
\end{proof}

Beside $S_1$ and $S_2$ we define for a partition $Z=\{t_0,\ldots,t_s\}$ of the interval $[0,t]$
\begin{align*}
	\tilde S_1(a,b,t,Z)&:= \sum_{l=0}^{s-1} \big[\lmr(a(t_{l+1}),b(t_{l+1})) - 
		\lmr(a(t_{l}),b(t_{l+1}))\big],\\
	\tilde S_2(a,b,t,Z)&:= \sum_{l=0}^{s-1} \big[\lmr(a(t_{l+1}),b(t_{l+1})) - 
		\lmr(a(t_{l+1}),b(t_l))\big].
\end{align*}
If $g'_t(0)=e^t$ holds for every $t\in[0,1]$, then it is easy to see that $S_1(a,b,t,Z)+
\tilde S_2(a,b,t,Z)=t$ and $S_2(a,b,t,Z) + \tilde S_1(a,b,t,Z)=t$. 
By Proposition \ref{Pro:ProCalLam} we see
\[
	\lim_{|Z|\rightarrow 0} \tilde S_1(a,b,Z,t) = t-c_2(t)=c_1(t),\quad
	\lim_{|Z|\rightarrow 0} \tilde S_2(a,b,Z,t) = t-c_1(t)=c_2(t).
\]

%% file: Input/Chapter3.tex
\section{Existence}
Now we are able to prove the existence part of Theorem \ref{The:ConCoeffi}. 
Recall that we have to show the existence of two strictly increasing homeomorphisms 
$u,v:[0,1]\to[0,1]$ and a $\lambda_0\in(0,1),$
such that the Komatu-Loewner equation for the slits $\gamma_1\circ u$ and 
$\gamma_2\circ v$ is satisfied for all $t\in[0,1]$ with $\lambda_1(t)= \lambda_0$ 
and $\lambda_2(t)=1-\lambda_0$ and $\lmr(u(t),v(t))=t$ for all $t\in[0,1].$

The proceeding of this proof is as follows.
\begin{enumerate} 
	\item First of all we will use a Bang-Bang method introduced in \cite{RothSchl} to construct 
		two sequences $(u_n)_{n\in\IN}$ and $(v_n)_{n\in\IN}$ of increasing self-homeomorphisms 
		of $[0,1]$.
	\item By using a diagonal argument on $u_n$ and $v_n$ we will find two subsequences 
		$(u_n^*)_{n\in\IN}$ and $(v_n^*)_{n\in\IN}$ which converge pointwise on a dense set 
		$S\subset [0,1]$ to increasing functions $u$ and $v$ respectively. 
		The functions $u$ and $v$ can be extended to continuous functions 
		defined on $[0,1]$, with $u(1)=1=v(1)$. 
		Furthermore, we will get $\lambda_0\in [0,1]$ by the construction of $u$ and $v$.
	\item Next we will derive a connection between the sum $S_1(u_n^*,v_n^*,t,Z)$ 
		and the sum $S_1(u,v,t,Z)$ for a given partition $Z$ of the interval $[0,t]$.
	\item Moreover, we will find a connection between $S_1(u_n^*,v_n^*,t,Z)$ and $\lambda_0$.
	\item By combining these results we will find $S_1(u,v,t,Z)\rightarrow \lambda_0 t$ if 
		$|Z|\rightarrow 0$. Furthermore, as a consequence of this, we will find $\lambda_0\in(0,1)$.
	\item Next will show that $u$ and $v$ are strictly increasing, i.e. both functions are 
		increasing self-homeomorphisms of $[0,1].$ 
	\item Finally we will obtain the Komatu-Loewner-Equation with constant coefficients 
		$\lambda_0$ and $1-\lambda_0$ for the parametrizations $u$ and $v$.
\end{enumerate}
\begin{proof}[Proof of Theorem \ref{The:ConCoeffi} (Existence)]
${}$
	\begin{selflist}
		\item To construct $u_n$ and $v_n$, we first extend both $\gamma_1$ and $\gamma_2$ to an interval 
			$[0,T^*],$ $T^*>1,$ such that $\gamma_1[0,T^*]$ and $\gamma_2[0,T^*]$ are still 
			disjoint slits and $\lmr(T^*,0)\geq1,$ $\lmr(0,T^*)\geq 1.$
			Let $n\in \IN$ and $\lambda\in[0,1]$.
			We let $t_{0,n}=\tau_{0,n}=0$ and for $k\in\{1,\ldots, n\}$ we define $t_{k,n}>0$ 
			and $\tau_{k,n}>0$ recursively as the unique values with 
			$$ 
				\lmr(t_{k,n}, \tau_{k-1,n})-\lmr(t_{k-1,n},\tau_{k-1,n})=
				\frac{\lambda}{n}, \quad \lmr(t_{k,n}, \tau_{k,n})-\lmr(t_{k,n}, \tau_{k-1,n})=
				\frac{1-\lambda}{n}. 
			$$

			Since $(t,\tau)\mapsto\lmr(t,\tau)$ is strictly increasing in both variables,  
			see Lemma \ref{Lem:PropLMR} b), we get
			\begin{align*}
				\lmr(t_{n,n},\tau_{n,n})&=1 \le \lmr(T^*,0) < \lmr(T^*,\tau_{n,n})\\
				\lmr(t_{n,n},\tau_{n,n})&=1 \le \lmr(0,T^*) < \lmr(t_{n,n},T^*).
			\end{align*}
			Consequently $t_{n,n},\tau_{n,n}\le T^*$.\\
			Furthermore, note that the values $t_{k,n}=t_{k,n}(\lambda)$ and 
			$\tau_{k,n}=\tau_{k,n}(\lambda)$ depend continuously on $\lambda$:
			This follows easily by induction and the continuity  
			and strict monotonicity of the function
			$(t,\tau)\mapsto \lmr(t,\tau)$, see Lemma \ref{Pro:Monfunf}.
			Consequently, for every $n\in\IN$, we can find a value $\lambda_n\in(0,1)$ with 
			$t_{n,n}(\lambda_n)=1$.
			Now we define a sequence of functions $u_n:[0,1]\rightarrow [0,t_{n,n}]$ and
			$v_n:[0,1]\rightarrow [0,\tau_{n,n}]$. Define
			\[
				u_n\Big(\frac{k}{2^n}\Big) := t_{k,2^n}(\lambda_{2^n}), \quad 
				v_n\Big(\frac{k}{2^n}\Big) := \tau_{k,2^n}(\lambda_{2^n})
			\]
			for all $k=0,\ldots,2^n$. The values of $u_n$ and $v_n$ between the supporting 
			points are defined by linear interpolation. An immediate consequence of this 
			construction is
			\begin{align} \label{Equ:1}
				\lmr\bigg(u_n\Big(\frac{k}{2^n}\Big),v_n\Big(\frac{k}{2^n}\Big) \bigg)
				=\lmr\big( t_{k,2^n}(\lambda_{2^n}),\tau_{k,2^n}(\lambda_{2^n}) \big)=\frac{k}{2^n}.
			\end{align}
		\item 
			Since $\lambda _{2^n}$ is bounded, we find a subsequence 
			$(m_{k,0})_{k\in\N}$ 
			such that $(\lambda_{2^{m_{k,0}}})_{k\in\N}$ is convergent with 
			the limit $\lambda_0\in[0,1]$.
			Next we set
			\[
				S:=\bigcup_{n=1}^\infty S_n, \quad  
				S_n:=\Big\{ \frac{k}{2^n}\, \big|\, k=0,\ldots, 2^n\Big\}.
			\]
			$S$ is a dense and countable subset of $[0,1]$. 
			Denote by $a:\IN\rightarrow S$ a bijective mapping.\\
			Since the sequences $(u_{m_{k,0}}(a_1))_{k\in\IN}$ and  $(v_{m_{k,0}}(a_1))_{k\in\IN}$ 
			are bounded (by $T^*$),	we find a subsequence $(m_{k,1})_{k\in\IN}$ 
			of $(m_{k,0})_{k\in\IN}$, so that $(u_{m_{k,1}}(a_1))_{k\in\IN}$ and
			$(v_{m_{k,1}}(a_1))_{k\in\IN}$ are convergent.
			
			Inductively, we define $(m_{k,l})_{k\in\N}$, $l\in\N$, to be a subsequence of 
			$(m_{k,l-1})_{k\in\N}$ such that $(u_{m_{k,l}}(a_l))$ and $(v_{n_{k,l}}(a_l))$ 
			are convergent.

			Consequently we can define sequences
			$u_n^*:= u_{m_{n,n}}$ and $v_n^*:= v_{m_{n,n}}$ which are (pointwise) 
			convergent in $S$.
			We denote by $u$ and $v$ the limit function, i.e.
			\[
				u(t):= \lim_{n\rightarrow \infty} u^*_n(t),\quad 
				v(t):= \lim_{n\rightarrow \infty} v^*_n(t)
			\]
			for all $t\in S$. Moreover we set $\lambda_n^*:=\lambda_{2^{m_{n,n}}}$
			and $S^*_n:=S_{m_{n,n}}$.
			By using equation (\ref{Equ:1}) we get 
			$\lmr\big(u_n^*(t),v_n^*(t)\big)=t$ for $t\in S$ if $n$ is big 
			enough. Consequently we find by using Lemma \ref{Pro:ComConf} a)
			\begin{align} \label{Equ:2}
				\lmr\big(u(t),v(t)\big)= 
				\lim_{n\rightarrow\infty} \lmr\big(u_n^*(t),v_n^*(t)\big) = t
			\end{align}
			for all $t\in S$.
			Furthermore, since $t\mapsto u_n^*(t)$ and $t\mapsto v_n^*(t)$ are strictly
			increasing, the functions $t\mapsto u(t)$ and $t\mapsto v(t)$ are increasing too.
			Moreover $u$ and $v$ can be extended in a continuous and unique way to [0,1]. 
			To see this, let $t_0\in(0,1)$ and define 
			\[
				t_1:=\lim_{t\nearrow t_0 \atop t\in S} u(t),\quad
				t_2:=\lim_{t\searrow t_0 \atop t\in S} u(t),\quad
				\tau_1:=\lim_{t\nearrow t_0 \atop t\in S} v(t),\quad
				\tau_2:=\lim_{t\searrow t_0 \atop t\in S} v(t).
			\]
			 Thus we find by Lemma \ref{Pro:ComConf}  a) and equation (\ref{Equ:2})
			\[
				\lmr(t_1,\tau_1)=\lim_{t\nearrow t_0\atop t\in S} \lmr(u(t),v(t)) = t_0
				=\lim_{t\searrow t_0\atop t\in S} \lmr(u(t),v(t)) = \lmr(t_2,\tau_2).
			\]
			Since $(t,\tau)\mapsto\lmr(t,\tau)$ is strictly increasing in both variables
			and $t_1\le t_2$ and $\tau_1\le\tau_2$, we find $t_1=t_2$ and $\tau_1=\tau_2$. 
			If $t_0\in\{0,1\}$ we can argue in the same way, so $t\mapsto u(t)$ and 
			$t\mapsto v(t)$ are continuous in $[0,1]$.
			Summarizing, $u$ and $v$ are continuous and increasing in [0,1] with $u(1)=1$ and 
			$v(1)=1$.
			For later use we define $h^{[n]}_{t,\tau}:=f_{u^*_n(t),v^*_n(\tau)}$ and 
			$h_{t,\tau}:=f_{u(t),v(\tau)}$.
		\item Next we show that for every fixed $\epsilon >0$, fixed $t\in S$ and a fixed
			partition $Z\subset S$ of the interval $[0,t]$, there exists an $n_0\in\IN$ so that 
			\[
				|S_1(u^*_n,v^*_n,t,Z) - S_1(u,v,t,Z)| < \epsilon
			\]
			holds for all $n\ge n_0$, where $Z=\{t_0,t_1,\ldots, t_s\}$. 
			
			Fix $\epsilon>0$.
			As the function $(t,\tau)\mapsto \lmr(t,\tau)$ is (uniformly) continuous in 
			$[0,T^*]^2$ by Lemma \ref{Pro:ComConf} a), there exists $\delta >0$ such that 
			\begin{align*} 
				|\lmr(\lt,\ltau)-\lmr(\gt,\gtau)|<\frac{\epsilon}{2s} 
				\quad\text{whenever} \quad |\lt-\gt|,\, |\ltau-\gtau|<\delta.
			\end{align*}
			Since $Z\subset S$, we find an $n_0\in\IN$ so that $|u_n^*(t_l)-u(t_l)|$,
			$|v_n^*(t_l)-v(t_l)|<\delta$ holds for all $l=0,\ldots, s$ and $n\ge n_0$.
			Consequently we find 
			\begin{align*}
				|&S_1(u^*_n,v^*_n,t,Z) - S_1(u,v,t,Z)|\\
				&= \Big|\sum_{l=0}^{s-1} \lmr(h^{[n]}_{t_{l+1},t_l})- \lmr(h^{[n]}_{t_l,t_l})
					+ \sum_{l=0}^{s-1} \lmr(h_{t_{l+1},t_l})-\lmr(h_{t_l,t_l})\Big|\\
				&\le \sum_{l=0}^{s-1} \big|\lmr(h^{[n]}_{t_{l+1},t_l})-\lmr(h_{t_{l+1},t_l})\big|
					+ \sum_{l=0}^{s-1} \big|\lmr(h^{[n]}_{t_l,t_l})-\lmr(h_{t_l,t_l})\big|
					\le 2 s \frac{\epsilon}{2s} = \epsilon.
			\end{align*}
		\item For now we fix $t\in S$. We show that for all $\epsilon >0$ we find 
			a $\mu>0$ so that for all partitions $Z\subset S$ of $[0,t]$ with $|Z|<\mu$ 
			there exists an $m_0 \in \IN$ so that for all $n\ge m_0$ we have
			\[
				|S_1(u^*_n,v^*_n,t,Z)-\lambda_0 t|<\epsilon.
			\]
			
			Let $\epsilon>0$. Then there exists $\delta>0$ such that the inequality from 
			Lemma \ref{Pro:DifQuoAbs2} c) holds. Since  the functions $t\mapsto u(t)$ 
			and $t\mapsto v(t)$ are (uniformly) continuous we get
			\begin{align*}
				\exists \mu>0:\; |\lt-\gt|<\mu \Rightarrow
				|u(\lt)-u(\gt)|, |v(\lt)-v(\gt)|<\frac{\delta}{2}.
			\end{align*}
			Denote by $Z=\{t_0,\ldots,t_s\}$ a partition of $[0,t]$ with $|Z|<\mu$ and 
			$Z\subset S$. 
			Then we find an $m_0\in\IN$ with $Z\subset S^*_n$, $t\in S^*_n$ and
			\[
				|u_n^*(t_l)-u(t_l)|,\, |v_n^*(t_l)-v(t_l)|< \frac{\delta}{4} 
			\]
			for all $n\ge m_0$ and all $l=0,\ldots, s$.
			As a consequence we get
			\begin{multline*}
				|u^*_n(t_{l+1})-u^*_n(t_l)|\\
					\le |u^*_n(t_{l+1})-u(t_{l+1})| + |u(t_{l+1})-u(t_l)| + |u(t_l)-u^*_n(t_l)|< 
					\frac{\delta}{4} + \frac{\delta}{2} + \frac{\delta}{4} = \delta
			\end{multline*}
			for all $n\ge m_0$ and all $l=0,\ldots, s$. In an analog way we get 
			$|v^*_n(t_{l+1})-v^*_n(t_l)|<\delta$.
			Next, we set $S^*_n(t):=S^*_n\cap[0,t]$. 
			$S^*_n(t)$ is a partition of the interval 
			$[0,t]$ and we write $S^*_n(t)=\{t^*_0,\ldots, t^*_{s^*}\}$.
			\begin{align*}
				|\lambda^*_n t-&S_1(u^*_n,v^*_n,t,Z)| = |S_1(u^*_n,v^*_n,t,S^*_n(t)) - 
					S_1(u^*_n,v^*_n,t,Z)|\\					
				&=\sum_{j=0}^{s^*-1}\big|[\lmr(h^{[n]}_{t^*_{j+1},t^*_j})-
					\lmr(h^{[n]}_{t^*_j,t^*_j})]-
					[\lmr(h^{[n]}_{t^*_{j+1},\phi(t^*_j)})-\lmr(h^{[n]}_{t^*_j,\phi(t^*_j)})]\big|
			\end{align*}
			where $\phi(t^*_j):=t_l$ if $t^*_j\in[t_l,t_{l+1})$ with $j=0,\ldots,s^*$ and 
			$l=0,\ldots,s$. Since $|v^*_n(t_{l+1})-v^*_n(t_l)|<\delta$, we have 
			$|v^*_n(\phi(t_l))-v^*_n(t_l)|<\delta$ for all $n\ge m_0$ and all 
			$l=0,\ldots, s$. Thus we get
			\begin{align*}
				|\lambda^*_n t-&S_1(u^*_n,v^*_n,t,Z)|\\
				&= \sum_{j=0}^{s^*-1} |\lmr(h^{[n]}_{t^*_{j+1},t^*_j})-\lmr(h^{[n]}_{t^*_j,t^*_j})|
					\cdot \left|1-\frac{\lmr(h^{[n]}_{t^*_{j+1},\phi(t^*_j)})-
					\lmr(h^{[n]}_{t^*_j,\phi(t^*_j)})}{\lmr(h^{[n]}_{t^*_{j+1},t^*_j})-
					\lmr(h^{[n]}_{t^*_j,t^*_j})} \right|
			\end{align*}
			Since $|u^*_n(t^*_{j+1})-u^*_n(t^*_j)|<\delta$ and 
			$|v^*_n(\phi(t^*_j))-v^*_n(t^*_j)|<\delta$ for all $n\ge m_0$ and all 
			$l=0,\ldots, s$, we have by Lemma \ref{Pro:DifQuoAbs2} c)
			\[
				|\lambda^*_n t-S_1(u^*_n,v^*_n,t,Z)| \le \epsilon 
				\sum_{j=0}^{s^*-1} \big(\lmr(h^{[n]}_{t^*_{j+1},t^*_j})-
					\lmr(h^{[n]}_{t^*_j,t^*_j})\big)\le\epsilon
			\]
			for all $n\ge m_0$. The last inequality can be proven by using the monotonicity of
			$(t,\tau)\mapsto\lmr(t,\tau)$ as follows
			\begin{align*}
				\sum_{j=0}^{s^*-1} &\big(\lmr(h^{[n]}_{t^*_{j+1},t^*_j})-
					\lmr(h^{[n]}_{t^*_j,t^*_j})\big)\\[-0.5\baselineskip]
				&\le \sum_{j=0}^{s^*-1} \big(\lmr(h^{[n]}_{t^*_{j+1},t^*_j})-
					\lmr(h^{[n]}_{t^*_j,t^*_j})\big)
				+\sum_{j=0}^{s^*-1} \big(\lmr(h^{[n]}_{t^*_{j+1},t^*_{j+1}})-
					\lmr(h^{[n]}_{t^*_{j+1},t^*_j})\big)\\
				&= \lmr(h^{[n]}_{t^*_{s^*},t^*_{s^*}}) - \lmr(h^{[n]}_{t^*_0,t^*_0}) = 
					\lmr(h^{[n]}_{t,t}) = t \le 1
			\end{align*}
			The assertion follows now, since $\lambda^*_n$ converges to $\lambda_0$.
		\item If we put 3) and 4) together we find for every $\epsilon>0$ a $\mu>0$ so that for all
			partitions $Z\subset S$ of the interval $[0,t]$ with $|Z|<\mu$ the 
			inequality 
			\begin{align} \label{Equ:4}
				|S_1(u,v,t,Z)-\lambda_0 t|<\epsilon
			\end{align}
			holds.
			
			As a consequence of this, we will show next that $\lambda_0\ne 0$. \\
			So let us assume the opposite, namely $S_1(u,v,1,Z)\rightarrow 0$ for 
			$|Z|\rightarrow 0$, where $Z\subset S$ denotes a partition of the interval $[0,1]$.
			Let $\delta>0$ such that the inequality from Proposition \ref{Pro:DifQuoAbs2} c) holds for  
			$\epsilon=\frac{1}{2}$. Since $t\mapsto u(t)$ and $t\mapsto v(t)$ are
			(uniformly) continuous in $[0,1]$, we get
			\[
				\exists \mu>0:\, |\lt-\gt|<\mu \Rightarrow |u(\lt)-u(\gt)|,\, 
				|v(\lt)-v(\gt)|<\delta
			\]	
			Denote by $Z:=\{t_0,\ldots,t_s\}\subset S$ a partition of the interval 
			$[0,1]$ with $|Z|<\mu$.
			Consequently, we have for all $n\in\IN$ with $Z\subset S_n:=\{t^*_0,\ldots,t^*_{2^n}\}$ 
			as before
			\begin{align*}
				0<\sigma:=S_1(u,v,1,Z)=\sum_{j=0}^{2^n-1} \lmr(h_{t^*_{j+1},\phi(t^*_j)})-
				\lmr(h_{t^*_j,{\phi(t^*_j)}}),
			\end{align*}
			where $\phi(t^*_j):=t_l$ if $t^*_j\in[t_l,t_{l+1})$ with $j=0,\ldots,2^n$ and 
			$l=0,\ldots,s$. Hence $|v(\phi(t^*_l))-v(t^*_l)|<\delta$, so we get by 
			Lemma \ref{Pro:DifQuoAbs2} c)
			\begin{align*}
				S_1(u,v,1,Z) \le (1+\epsilon)\sum_{j=0}^{2^n-1} \lmr(h_{t^*_{j+1},t^*_j})-
				\lmr(h_{t^*_j,t^*_j}) = (1+\epsilon)S_1(u,v,1,S_n)
			\end{align*}
			Thus we get $\frac{2}{3}\sigma\le S_1(u,v,1,S_n)$, so $S_1(u,v,1,S_n)$ does not tend to
			zero as $n$ tends to infinity. This is a contradiction, so $\lambda_0\ne 0$.
		\item Next we will show, as another consequence of equation (\ref{Equ:4}), that the 
			function $t\mapsto u(t)$ is strictly increasing.
			For this purpose, we assume the opposite. Let $t_1<t_2$ with
			$u(t_1)=u(t_2)$. Without loss of generality we can assume $t_1,t_2\in S$.
			So we find $n_0\in\IN$ with $t_1,t_2\subset S_n$ for all $n\ge n_0$. 
			Let $\epsilon:=\frac{1}{2}(t_2-t_1)\lambda_0$. Then there exist $\mu_1,\mu_2>0$ such 
			that
			\[
				|S_1(u,v,t_1,Z_1)-\lambda_0 t_1|<\epsilon\quad \text{and} \quad
				|S_1(u,v,t_2,Z_2)-\lambda_0 t_2|<\epsilon
			\]
			for all partitions $Z_1$ of $[0,t_1]$ and $Z_2$ of $[0,t_2]$ with $|Z_1|<\mu_1$
			and $|Z_2|<\mu_2$. Consequently we find an $m_0\ge n_0$ with 
			$|S_{m_0}|<\min(\mu_1,\mu_2)$, so we get
			\[
				\lambda_0(t_2-t_1) < 2\epsilon,
			\]
			since $S_1(u,v,t_1,S_{m_0}(t_1)) = S_1(u,v,t_2,S_{m_0}(t_2))$. 
			This is a contradiction to $\lambda_0\ne 0$, so
			$t\mapsto u(t)$ needs to be strictly increasing.
			
			By applying the steps 3) - 6) to the second slit with $\tilde S_2$ instead of $S_1$, we 
			get $\lambda_0\ne 1$ and the strict monotonicity of $t\mapsto v(t)$.
		\item Since $u$ and $v$ are strictly increasing self-homeomorphisms of $[0,1]$, 
			we can apply Proposition \ref{Pro:ProCalLam} (for $a=u$ and $b=v$) 
			to get Lipschitz continuous and increasing functions $c_1$ and $c_2$. 
			By equation (\ref{Equ:4}) we see
			\[
				c_1(t)=\lambda_0 t,\quad c_2(t)=(1-\lambda_0)t
			\]
			for all $t\in S$. Since $S$ is dense in $[0,1]$ this relation holds for 
			all $t\in[0,1]$. \\
			
			Finally, from Proposition \ref{Pro:ProCalLam} it follows that for every $z\in \Omega_1$ the differential equation			
			\[\label{eq:twoslits}
				\dot h_t(z) = h_t(z)\big( \lambda_0 \Phi(\xi_1(t),h_t(z),D_t) + (1-\lambda_0)
					\Phi(\xi_2(t),h_t(z),D_t)\big)
			\]
			holds for all $t\in[0,1]$, where $h_t:=h_{t,t}$ and $\xi_j$ is the driving function for $\Gamma_j.$
			
			As $t\mapsto \gamma_1(u(t))$ and $t\mapsto \gamma_2(v(t))$ are continuous,
			the continuity of the driving functions $t\mapsto \xi_1(t)$ and 
			$t\mapsto \xi_2(t)$ follows directly from Proposition 7 in \cite{BoehmLauf}.
	\end{selflist}
\end{proof}

%% file: Input/Chapter4.tex
\section{Uniqueness}
\begin{proof}[Proof of Theorem \ref{The:ConCoeffi} (Uniqueness)] 
	Let us denote by $u_1, u_2, v_1,v_2$ increasing self-homeomorphisms of $[0,1]$ with 
	$\lmr(u_1(t),v_1(t))=\lmr(u_2(t),v_2(t))=t$ for all $t\in[0,1]$ such that the functions 
	$h_t:= f_{u_1(t),v_1(t)}$ and $g_t:=f_{u_2(t),v_2(t)}$ satisfy the differential equations 
	\begin{eqnarray*}
		\dot h_t(z) &=& h_t(z)\big( \lambda_1 \Phi(\xi_1(t),h_t(z),D_t) + (1-\lambda_1)
					\Phi(\xi_2(t),h_t(z),D_t)\big),\\
		\dot g_t(z) &=& g_t(z)\big( \lambda_2 \Phi(\zeta_1(t),g_t(z),E_t) + (1-\lambda_2)
					\Phi(\zeta_2(t),g_t(z),E_t)\big)
	\end{eqnarray*} 
	for all $t\in[0,1]$ with coefficients $0<\lambda_2\le\lambda_1<1$ and 
	continuous driving functions $\xi_1,\xi_2,\zeta_1,\zeta_2,$ 
	where $\xi_j$ and $\zeta_j$ correspond to the slit $\Gamma_j.$
	The continuity of the driving functions is an immediate consequence of Proposition 7 
	in \cite{BoehmLauf}, since $u_1, u_2, v_1,v_2$ are increasing self-homeomorphisms of $[0,1]$.
	By $D_t$ and $E_t$ we denote 
	the circular slit disks that are uniquely determined by $u_1, v_1$ and $u_2, v_2$ respectively.
	\begin{selflist}
		\item First of all we will show $\lambda_1=\lambda_2$, so let us assume 
			$\lambda_1>\lambda_2$.
			
			The differential equations immediately imply $h_t'(0)=g_t'(0)=e^t$ for all $t\in[0,1]$ 
			and by Proposition \ref{Pro:ProCalLam} we get Lipschitz-continuous functions 
			$c_1$ and $c_2$ for the case $a=u_1$ and $b=v_1$. 
			These functions are differentiable a.e. and it holds 
			\[
				\dot c_1(t) = \lambda_1, \quad 
				\dot c_2(t) = 1-\lambda_1, \quad \text{a.e. in} \; [0,1].
			\]
			 This is based on the fact, that the functions 
			\[
				z\mapsto \Phi(\xi_1(t),h_t;D_t),\quad
				z\mapsto \Phi(\xi_2(t),h_t;D_t)
			\]
			are for fixed $t$ linear independent and equation (\ref{KLE2}).\\
			Analogously, Proposition \ref{Pro:ProCalLam} gives us two Lipschitz-continuous 
			functions $d_1$ and $d_2$ for the case $a=u_2$ and $b=v_2$ with 
			\[ 
				\dot d_1(t) = \lambda_2, \quad  \dot d_2(t)= 1-\lambda_2 \quad \text{a.e. in} 
				\;[0,1]. 
			\]
			 
			An immediate consequence of the Lipschitz continuity is 
			$$	
				c_1(t)=\lambda_1 t,\; d_1(t)=\lambda_2 t,\; 	
				c_2(t)=(1-\lambda_1)t,\; d_2(t)=(1-\lambda_2)t \quad 
				\text{for all} \; t\in[0,1].
			$$
			Furthermore we set
			\[
				x_j(t):=\lmr\big(u_j(t),v_j(0)\big) = \lmr\big(u_j(t),0\big)
			\]
			for $j=1,2$. Denote by $0<t_0\le 1$ the first positive time when $u_1(t_0)=u_2(t_0)$.
			Consequently $v_1(t_0)=v_2(t_0)$ and $x_1(t_0)=x_2(t_0)$ by normalization and the
			monotonicity of $(t,\tau)\mapsto \lmr(t,\tau)$ in each variable.
			Since 
			$$
				\dot x_1(0) = \lambda_1 > \lambda_2 = \dot x_2(0)
			$$ 
			by Proposition \ref{Pro:ProCalLam}, we have $x_1(t)>x_2(t)$
			and as a consequence $u_1(t)>u_2(t)$ for all $t\in (0,t_0)$. Consequently we have also 
			$\lmr\big(u_1(t),v_1(t_0)\big)>\lmr\big(u_2(t),v_2(t_0)\big)$ for all 
			$t\in(0,t_0)$. Thus we get
			\begin{multline*} 
				\lmr\big(u_2(t),v_2(t_0)\big) < \lmr\big(u_1(t),v_1(t_0)\big) < 
				\lmr\big(u_1(t_0),v_1(t_0)\big) = \lmr\big(u_2(t_0),v_2(t_0)\big) = t_0
			\end{multline*}
			if $t<t_0$. This implies
			\[
				\frac{\lmr\big(u_1(t_0),v_1(t_0)\big)-\lmr\big(u_1(t),v_1(t_0)\big)}{t_0-t} < 
				\frac{\lmr\big(u_2(t_0),v_2(t_0)\big)-\lmr\big(u_2(t),v_2(t_0)\big)}{t_0-t}
			\]
			for all $t<t_0$. If $t$ tends to $t_0$ we get $\lambda_1\le \lambda_2$ 
			by Proposition \ref{Pro:ProCalLam}. This is a contradiction, so $\lambda_1=\lambda_2=:\lambda$.
		\item Next we prove the uniqueness of the parametrizations $u$ and $v$, i.e. we show
			$u_1\equiv u_2$. By using the result from the first part, we have
			\[
				c_1(t)=d_1(t) = \lambda t,\quad 
				c_2(t)=d_2(t) = (1-\lambda)t,
			\]
			for all $t\in[0,1]$.\\
			
			Again, we extend both $\gamma_1$ and $\gamma_2$ to an interval 
			$[0,T^*],$ $T^*>1,$ such that $\gamma_1[0,T^*]$ and $\gamma_2[0,T^*]$ are still 
			disjoint slits and $\lmr(T^*,0)\geq1,$ $\lmr(0,T^*)\geq 1.$ 
			Let $t\in(0,1)$ be fixed. Next we denote by $t_0=0<t_1<\ldots<t_n$
			and $\tilde t_0=0<\tilde t_1<\ldots< \tilde t_n$ the unique values such that
			\begin{eqnarray*}
				&&\lmr\big(u_1(t_{l+1}),v_1(t_l)\big) - \lmr\big(u_1(t_l),v_1(t_l)\big)\\ &=& 
				\lmr\big(u_2(\tilde t_{l+1}),v_2(\tilde t_l)\big) - 
					\lmr\big(u_2(\tilde t_l),v_2(\tilde t_l)\big) = 
				\frac{\lambda}{n} t
			\end{eqnarray*}
			holds. 
			By induction, it is easy to see that
			$$
				u_1(t_l)= u_2(\tilde t_l) \quad \text{and} \quad 
				v_1(t_l) = v_2(\tilde t_l) \quad \text{for all}\;\, l=1,\ldots,n.
			$$
			Furthermore, the values $|t_{l+1}-t_l|$ and $|\tilde t_{l+1}-\tilde t_l|$ 
			$(l=0,\ldots, n-1)$ become arbitrary small, if $n$ is big enough.
			Since $t_n$ and $\tilde t_n$ are bounded by $T^*$, we find convergent subsequences 
			$(t_{n_j})_{j\in\IN}$ and $(\tilde t_{n_j})_{j\in\IN}$. 
			The limit of both sequences needs to be $t$ since the sum
			\[ 
				\sum_{l=0}^{n-1} \lmr\big(u_j(t_{l+1}),v_j(t_l)\big) - 
				\lmr\big(u_j(t_l),v_j(t_l)\big),
			\]
			converges to $c_1(t)=d_1(t)=\lambda t$ for each $j=1,2$. 
			Finally, as a consequence of the continuity of $t\mapsto u_j(t)$, we get 
			$u_1(t)=u_2(t)$, so the proof is complete.
	\end{selflist}
\end{proof}

%% file: mainfile.bbl
\providecommand{\bysame}{\leavevmode\hbox to3em{\hrulefill}\thinspace}
\providecommand{\MR}{\relax\ifhmode\unskip\space\fi MR }
\providecommand{\MRhref}[2]{%
  \href{http://www.ams.org/mathscinet-getitem?mr=#1}{#2}
}
\providecommand{\href}[2]{#2}
\begin{thebibliography}{10}

\bibitem{BauerFriedrichCSD}
Robert~O. Bauer and Roland~M. Friedrich, \emph{{On radial stochastic Loewner
  evolution in multiply connected domains.}}, J. Funct. Anal. \textbf{237}
  (2006), no.~2, 565--588 (English).

\bibitem{BauerFriedrichCBC}
\bysame, \emph{On chordal and bilateral {SLE} in multiply connected domains},
  Math. Z. \textbf{258} (2008), no.~2, 241--265. \MR{2357634 (2009b:60292)}

\bibitem{BoehmLauf}
Christoph Boehm and Wolfgang Lauf, \emph{A {K}omatu-{L}oewner-{E}quation for
  {M}ultiple {S}lits}, preprint.

\bibitem{ConwayII}
John~B. Conway, \emph{{Functions of one complex variable. II.}}, {Graduate
  Texts in Mathematics. 159. New York, NY: Springer-Verlag. xvi, 394 p. DM
  88.00; \"oS 686.40; sFr 84.50 }, 1995 (English).

\bibitem{Goodman}
Gerald~S. Goodman, \emph{Control {T}heory in {T}ransformation {S}emigroups},
  Geometric Methods in System Theory, Proceedings NATO Advanced Study Institute
  (1973), 215--226.

\bibitem{KomatuZweifach}
Yusaku Komatu, \emph{{Untersuchungen \"uber konforme {A}bbildung von zweifach
  zusammenh\"angenden {G}ebieten.}}, Proc. Phys. Math. Soc. Japan, III. Ser.
  \textbf{25} (1943), 1--42 (German).

\bibitem{Komatu}
\bysame, \emph{{On conformal slit mapping of multiply-connected domains.}},
  Proc. Japan Acad. \textbf{26} (1950), no.~7, 26--31 (English).

\bibitem{Loewner:1923}
K.~L{\"o}wner, \emph{Untersuchungen {\"u}ber schlichte konforme {A}bbildungen
  des {E}inheitskreises. {I}}, Mathematische Annalen \textbf{89} (1923), no.~1,
  103--121.

\bibitem{Prokhorov:1993}
D.V. Prokhorov, \emph{Reachable {S}et {M}ethods in {E}xtremal {P}roblems for
  {U}nivalent {F}unctions}, Saratov University, 1993.

\bibitem{RothSchl}
Oliver Roth and Sebastian Schlei{\ss}inger, \emph{The chordal {L}oewner
  equation for multiple slits}, preprint.

\bibitem{MR1776084}
Oded Schramm, \emph{Scaling limits of loop-erased random walks and uniform
  spanning trees}, Israel J. Math. \textbf{118} (2000), 221--288.

\end{thebibliography}
